\documentclass[10pt]{amsart}
\usepackage[latin9]{inputenc}
\usepackage{geometry}
\usepackage{verbatim}
\usepackage{amstext}
\usepackage{amsthm}
\usepackage{amssymb}
\usepackage{graphicx}
\usepackage[hidelinks=true]{hyperref}

\makeatletter

\newcommand{\R}{ {\mathbb{R}} }
\newcommand{\E}{ {\mathbb{E}} }
\renewcommand{\P}{ {\mathbb{P}} }
\newcommand{\D}{ {\mathbb{D}} }

\newcommand{\cR}{ {\mathcal{R}} }
\newcommand{\cH}{ {\mathcal{H}} }

\renewcommand{\and}{\quad\textrm{ and }\quad}

\newcommand{\C}{\mathcal{C}}

\usepackage{xcolor}

\newtheorem{theorem}{Theorem}[section]
\theoremstyle{plain}

\newtheorem{corollary}[theorem]{Corollary}

\newtheorem{lemma}[theorem]{Lemma}

\newtheorem{proposition}[theorem]{Proposition}
\newtheorem{remark}[theorem]{Remark}

\numberwithin{equation}{section}

\newcommand{\black}{\textcolor{black}}

\title{Non-uniqueness for reflected rough differential equations}

\author{Paul Gassiat}
\address{
Universit\'e Paris-Dauphine, PSL University, UMR 7534, CNRS, CEREMADE, 75016 Paris, France}
\email{gassiat@ceremade.dauphine.fr}

\begin{document}

\begin{abstract}
We give an example of a reflected differential equation which may have infinitely many solutions if the driving signal is rough enough (e.g. of infinite $p$-variation, for some $p>2$). For this equation, we identify a sharp condition on the modulus of continuity of the signal under which uniqueness holds. L\'evy's modulus for Brownian motion turns out to be a boundary case. We further show that in our example, non-uniqueness holds almost surely when the driving signal is a fractional Brownian motion with Hurst index $H < \frac{1}{2}$. The considered equation is driven by a two-dimensional signal with one component of bounded variation, so that rough path theory is not needed to make sense of the equation.\end{abstract}

\thanks{This work is partially supported by the ANR via the project ANR-16-CE40- 0020-01. The author would like to thank Joseph Lehec for a helpful discussion, and Cyril Labbé for useful comments. The author is also grateful to an anonymous referee for several remarks which helped to improve the clarity of the presentation.}

\maketitle


\section{Introduction}

We consider differential equations with normal reflection taking values in a closed domain $D \subset \R^d$ and driven by a signal $X$, which in general take the form
\begin{equation} \label{eq:RSDE}
dY_t = f(Y_t) dX_t + d\kappa_t, \;\;Y_0 = y_0,
\end{equation}
where the unknown is the pair $(Y,\kappa)$ which must satisfy the additional constraint
$$\forall t \geq 0, \;\; Y_t \in D, \;\;\; d\kappa_t =  1_{\{Y_t \in \partial D\}} n(Y_t) |d\kappa_t| $$
where $n(y)$ is an inner normal of $D$ at $y \in \partial D$. 
\smallskip

In the stochastic analysis literature, the driving signal $X$ is usually a (continuous) semimartingale, and the equation is understood in It\^o or Stratonovich sense. Existence and uniqueness of the solutions are then classical (see e.g. \cite{Tan79, LS84, Sai87}). In fact, in this context the difficult part is usually the existence, while the uniqueness is an almost immediate consequence of It\^o's formula, under some mild regularity assumption on $D$ (external ball condition).
\smallskip

However, these well-posedness results rely crucially on It\^o's calculus, and therefore are restricted to semimartingale signals. 
In contrast, Lyons' rough path theory \cite{Lyo98} provides a deterministic framework to define integrals (and solve differential equations) driven by signals $X$ of arbitrary low regularity (measured for instance by the index $p$ in the scale of $p$-variation spaces). 
The key idea of rough path theory is to lift $X$ to an enhanced object $\mathbf{X} = (X, \int X \otimes dX,\ldots)$ in a suitable metric space (depending on $p$) so that the solution of a differential equation driven by $X$ is then obtained as a \emph{continuous} function of $\mathbf{X}$
(we will not need rough path theory in this paper so we refrain from giving any more details).
In addition to the added robustness which is useful even when applied to the semimartingale framework, the flexibility of rough path theory means that it may be applied to a much broader class of random signals, such as for instance many Gaussian or Markovian processes, see e.g. \cite{FV} and references therein.
\smallskip

It is therefore interesting to understand to which extent a (rough) pathwise theory is possible for \eqref{eq:RSDE}. Let us summarize the results which are known so far. Existence results have been proven by Aida \cite{Aida15,Aida16} when $X$ is a rough path with finite $p$-variation ($p<3$), under essentially the same assumptions on $D$ as in the semimartingale case. In the Young case ($p<2$), uniqueness of solutions was obtained by Falkowski and S\l{}omi\'{n}ski \cite{FS15} (in the case $D = \R_+^{d_1} \times \R^{d_2}$) by a contraction mapping argument. The same result was then extended to mixed Young/semimartingale SDE by the same authors \cite{FS17}. In the rough case ($p <3$), uniqueness has been obtained in the one-dimensional case $D=\R_+$ by Deya et al. \cite{DGHT19}. Similar results to those mentioned above have also been obtained in the case when $D$ is allowed to depend on time, see e.g. \cite{FS15b,CMR17,RTT19}.
\smallskip

{\color{black}
However, the question of uniqueness in the case of rough signals ($p>2$) and multidimensional domains ($d \geq 2$) has so far remained open, and the main goal of this paper is to resolve it (negatively). For preciseness let us consider the following formal statement (for given integers $m, d \geq 1$ and scalar $p \geq 1$) :
\begin{align*}
\mbox{\textbf{Assertion (A)}}_{d,m,p} \;\;\; &\mbox{For any smooth domain $D \subset \R^d$, any rough path $X$ $\in$ $C_ {p-var}([0,1],\R^m)$, }\\
&\mbox{  any $f \in C^\infty \left( \R^d, L(\R^m,\R^d)\right)$ and any $y_0 \in \R^d$,}\\
&\mbox{ the equation \eqref{eq:RSDE} admits at most one solution $(Y,\kappa)$.}
\end{align*}

(Note that, if $p \geq 2$, in the above $X$ should be understood as $p$-rough path over a $\R^m$-valued path, with \eqref{eq:RSDE} being understood in the sense of rough integration.)

Then by the results mentioned in the previous paragraph, it is known that  \textbf{(A)}$_{d,m,p}$ holds if either $1 \leq p< 2$ \cite{FS15}, or if $d=1$ \cite{DGHT19}.  The main result of the present paper can be  summarized as follows.
\begin{theorem}
For any  $d \geq 2$, $m \geq 2$, $p >2$, the assertion {\rm\textbf{(A)}$_{d,m,p}$} is false.
\end{theorem}
}
We prove this result via a simple counter-example showing that an equation of the form \eqref{eq:RSDE} driven by a rough signal may have infinitely many solutions, even for smooth domains (in our case the domain is just $\R_+ \times \R$). The presented equation is (affine) linear, and since  the rough component of the driving signal is one-dimensional, we may define solutions by a Doss-Sussman representation, so that we actually do not need rough path theory in this paper. 
\smallskip

Our result shows that uniqueness may not hold for signals of finite $p$-variation with $p>2$, while it is known to hold for $p<2$, and it is natural to ask the exact regularity at which uniqueness breaks down. 
In the case of the equation considered in this paper, we obtain a precise answer in the form of
a necessary and sufficient condition on a modulus of continuity for uniqueness to hold for arbitrary signals of the corresponding regularity. Interestingly, Lévy's modulus of continuity for Brownian motion turns out to be a boundary case (namely, the power to which the logarithm appears in the modulus is critical). Since our counterexample consists of carefully chosen deterministic paths, we also discuss what happens when the driving signal in our equation comes from a probabilistic distribution. We focus on the case of fractional Brownian motion with Hurst index $H < \frac{1}{2}$, and we obtain that in that case uniqueness still does not hold (almost surely).
\smallskip

{\color{black}
Finally, let us mention that if the main result of the paper shows that regularity properties of $f$ and $X$ are not sufficient to establish uniqueness of solutions to \eqref{eq:RSDE} in the rough setting, one may still hope that finer properties of the vector fields and/or the driving signal suffice to restore uniqueness. We present several (loose) conjectures in this direction below.
}
\smallskip

The remainder of the paper is organized as follows. In Section 2 we present the counterexample, state our main results  and comment on them. In Section 3 we give the proofs of these results. The proof of some Gaussian estimates is delayed to Section 4.

\section{Main results}

The equation that we consider is written 

\begin{align}\label{eq:RRDE}
dZ_t = A Z_t d \lambda_t -  e_1d\gamma_t + e_1 dK_t  \;\;\;\; \mbox{ for } t\in[0,1],,  \\
Z \cdot e_1 \geq 0, \;\;\; dK = 1_{\{ Z\cdot e_1 = 0\}} |dK|  \notag
\end{align}
where the unknown $(Z,K)$ takes values in $\R^2 \times \R$, $(e_1,e_2)$ denotes the canonical basis,  
$$A = \left(\begin{array}{cc} 0&1\\1&0\end{array}\right),$$
$\lambda$ is a given scalar continuous path, and $\gamma$ is a nondecreasing scalar function.
\smallskip

Note that since $\lambda$ is the only component of the driving signal of unbounded variation in the above equation, we may use a Doss-Sussman representation to define solutions (in fact, since the equation is linear, this is just Duhamel's formula), see subsection \ref{sec:prel} below for precise statements. We will frequently make the abuse of notation of calling $Z$ itself the solution.

We also note that $(Z\equiv 0, K = \gamma)$ is a solution to \eqref{eq:RRDE}, so that to prove that uniqueness does not hold it will be enough to find solutions with $Z_0=0$ and non-null $Z$ component.  

\begin{theorem} \label{thm:ce1}
There exists $\lambda$ $\in$ $ \cap_{p>2} $$C_{p-var}([0,1], \R)$, and $\gamma$ continuous and increasing s.t. \eqref{eq:RRDE}
admits uncountably many distinct solutions on $[0,1]$ with $Z_0=0$, which are all non-null at positive times.
\end{theorem}

Let us describe how these solutions are obtained. The trajectories corresponding to the linear part of the equation (driven by $\lambda$) are given by hyperboles asymptotic to the $\{x=y\}$ line, and which cross the $y$ axis (i.e. the reflecting boundary) in the normal direction. On the other hand, in the part of the plane where the equations are constrained to live, the drift $- e_1 d\gamma$ pushes the solution $Z$ towards these hyperboles that are further away from the origin. The solutions from the Theorem above are then obtained by alternating intervals where $\lambda$ acts by moving $Z$ away from the $y$ axis along a small hyperbole arc and then $\gamma$ pushes $Z$ back to the $y$ axis, see Figure \ref{fig:1} below. One then sees that taking $\lambda$ of infinite $2$-variation, one may accumulate infinitely many such small intervals in such a way that  the solution may  escape from $0$ in finite time. The additional restriction that $\gamma$ must have finite total variation imposes a further constraint on how $\lambda$ must be chosen (actually, both constraints combined impose that $\lambda$ has infinite $\psi$-variation, for $\psi(r) = r^2 / \log(r^{-1})$, cf Lemma \ref{lem:logvar}).

\vspace{5mm}

We then focus on the case where $d\gamma=dt$ in \eqref{eq:RRDE}, which for convenience we rewrite below :

\begin{equation}\label{eq:RRDEdt}
dZ_t = A Z_t d \lambda_t - e_1 \; dt  + e_1 dK_t, \;\;\;
Z \cdot e_1 \geq 0, \;\;\; dK = 1_{\{ Z\cdot e_1 = 0\}} |dK|.
\end{equation}
\smallskip

We obtain a sharp criterion on the modulus of continuity of $\lambda$ so that the above admits a unique solution. We say that $\omega :[0,\infty) \to [0,\infty)$ is a \emph{modulus} if it is continuous, non-decreasing, subadditive (i.e. $\omega(a+b) \leq \omega(a) + \omega(b)$ for all $a,b \geq 0$), and satisfies $\omega(0)=0$. Given a modulus $\omega$, we let
$$\C_\omega = \big\{ f:[0,1] \to \R, \;\;\; \forall 0 \leq s < t \leq 1, \left|f(t) - f(s) \right| \leq \omega\left(|t-s|\right) \big\}.$$

\begin{theorem}  \label{thm:cedt}
Given a modulus $\omega$, let 
\begin{equation} \label{eq:deftheta}
\theta_{\omega}(\varepsilon) := \sup_{r \geq 0} \left( \omega(r) - \frac{r}{\varepsilon}\right).
\end{equation}
Then if $\theta_{\omega}$ satisfies Osgood's condition 
\begin{equation} \label{eq:Osg}
 \int_{0^+} \frac{dx}{\theta_{\omega}(x)} = +\infty, 
\end{equation}
the equation \eqref{eq:RRDEdt} admits a unique solution for any $\lambda$ in $\C_\omega$ and any initial condition $Z_0$.

On the other hand, if \eqref{eq:Osg} does not hold and in addition
\begin{equation} \label{eq:om2delta}
\limsup_{\delta \to 0} \frac{\omega(2 \delta)}{2 \omega(\delta)} < 1,
\end{equation}
then there exists $\lambda$ in $\C_\omega$ such that \eqref{eq:RRDEdt} admits multiple solutions with $Z_0=0$.
\end{theorem}
\smallskip

\begin{remark} \label{rem:Osgood}
Note that in the case of the H\"older modulus $\omega(r) = r^{\alpha}$ for $\alpha \in (0,1]$, one has that $\theta_{\omega}(\varepsilon) = C_{\alpha} \varepsilon^{\frac{\alpha}{1-\alpha}}$, so that \eqref{eq:Osg} holds if and only if $\alpha \geq \frac{1}{2}$ (recall that we already know from \cite{FS15} that equations are well-posed for $\lambda \in C^\alpha$ with $\alpha > \frac 1 2 $).

 In the case when  $\omega(r) = r^{1/2} \log(r^{-1})^\beta$, one has $\theta_{\omega}(\varepsilon) \sim C_{\beta} \varepsilon \log(\varepsilon^{-1})^{2\beta}$ as $\varepsilon \to 0$, so that \eqref{eq:Osg} holds  if and only if $\beta \leq 1/2$ (this is rather striking, since $\beta =1/2$ is exactly the case of L\'evy modulus  for Brownian motion !).
 
 We also remark that \eqref{eq:om2delta} is a rather mild assumption, for instance it is implied by the fact that \eqref{eq:Osg} does not hold if  $\omega$ is regularly varying at $0^+$\black{(indeed, if $\omega(r) = r^\alpha L(r)$ with $L$ slowly varying at $0$, then \eqref{eq:Osg} holds unless $\alpha \leq 1/2$, whereas $\alpha < 1$ implies \eqref{eq:om2delta})}.

\end{remark}

\black{
\begin{remark} \label{rmk:thm1vs2}
At first glance, Theorem \ref{thm:ce1} may seem like it is a consequence of Theorem \ref{thm:cedt}. However, it is not the case, since in Theorem \ref{thm:ce1} we are able to obtain uncountably many different solutions, whereas in Theorem \ref{thm:cedt}, in the non-uniqueness result we only construct a single nonnull solution to \eqref{eq:RRDEdt} (another one could be obtained by symmetry). We do not know if in that setting we could obtain multiple (three or more) solutions escaping immediately from $0$. (In addition, the solutions constructed in the proof of Theorem \ref{thm:ce1} are very explicit, so that for clarity's sake we believe it is helpful to separate this rather simple counter-example from the more technical computations needed in the proof of Theorem \ref{thm:cedt}  )
\end{remark}
}
\vspace{3mm}
Finally, we consider the probabilistic setting in which $\lambda$ is a fractional Brownian path. The results are as follows :

\begin{theorem}  \label{thm:fbm}
Let $0< H < \frac{1}{2}$ and $\P^H$ be the fBm measure of Hurst index $H$ on $C([0,1],\R)$. 

{\color{black}
(1) Let $\gamma : [0,1] \to \R$ be a piecewise-constant path of the form
\[ \gamma(t) = \sum_{0< t \leq t_k} x_k, \]
where $(t_k)_{k \geq 0}$ decreases to $0$ and $(x_k)_{k \geq 0}$ is a summable sequence satisfying
\[ t_{k+1} - t_{k} \sim_{k \to \infty} k^{-\alpha}, \;\;\; x_k \geq \exp(-k^\theta), \]
where
\[ \max\left(\frac{5}{12H} ,1 \right)< \alpha < \frac{1}{2H}, \;\;\;\; 0< \theta<1- 2 \alpha H. \]
}Then for $\P^H$-almost every $\lambda$, \eqref{eq:RRDE} admits infinitely many solutions. More precisely, there exists a family of processes $(Z^\eta)_{\eta>0}$, adapted w.r.t. the completed natural filtration, that solve \eqref{eq:RRDE} a.s., and such that  a.s. one has for $\eta_1, \eta_2 >0$, 
$$\lim_{t\to 0} \frac{\left|Z^{\eta_1}_t\right|}{\left|Z^{\eta_2}_t \right|} = \frac{\eta_1}{\eta_2}.$$

(2) {\color{black} For the equation with $\gamma(t) = t$}, for $\P^H$-a.e. path $\lambda$, there exists a family $(Z^t)_{t \in [0,1]}$ of functions which are  solutions to \eqref{eq:RRDEdt} and such that $Z^t(s) = 0$ if and only if $s \leq t$ (in particular, these solutions are all distinct).
\end{theorem}

 In Figure 2 below we plot (a numerical approximation of) the trajectory of a non-null solution starting from $0$ given by point (2) above, with $H=0.2$.

\begin{remark}
The families of solutions in (1) and (2) are different in two respects :
\begin{itemize}
\item In (1), the solutions are obtained as adapted processes, whereas in (2) they are not.This is what is usually referred to as the distinction between pathwise and path-by-path uniqueness, cf. e.g. \cite{Fla11}. This distinction comes from the method of proof followed in both cases (in the proof of (2) one uses a compactness argument and the chosen subsequence may depend on the path of $\lambda$, whereas in the proof of (1) we do not need to pass to a subsequence). It seems likely that actually the solutions in (2) could also be obtained as adapted processes.

\item In (2), while there are also uncountably many solutions, they only differ by the time at which they leave the "problematic point" (here, the origin), whereas in (1) there are infinitely many solutions leaving $0$ at the same time (this is the same distinction as between the results of Theorem \ref{thm:ce1} and \ref{thm:cedt}, cf Remark \ref{rmk:thm1vs2}).  We also do not know if in (2) we could obtain multiple solutions escaping immediately from $0$.
 \end{itemize}
\end{remark}

{\color{black}
\begin{remark} \label{rmk:conj}
While the main result of this paper implies that well-posedness of rough differential equations with reflection cannot hold in general, one may still hope that uniqueness holds under further restrictions. Let us discuss three additional conditions under which one may conjecture that uniqueness holds. 
\begin{enumerate}
\item {\rm \textbf{Regularity of the signal}.} In equation \eqref{eq:RRDEdt}, uniqueness holds when the driving path has Brownian regularity (as discussed in Remark \ref{rem:Osgood}). One may conjecture that this would still be true for a general equation of the form \eqref{eq:RSDE}, assuming that the rough path associated to $X$ is sufficiently regular, for instance if it has finite $\psi$-variation for suitable $\psi$. Note that the important class of \emph{Markovian} rough paths, as described in \cite[chapter 16]{FV}, have sample paths with similar variation regularity as semimartingales, so that such a result would allow to consider reflected equations driven by such rough paths. This would also imply that rough path-wise methods may be applied to classical reflected SDE, which might prove useful in certain contexts.
\item   {\rm \textbf{Non-degeneracy of the equation}.} In \eqref{eq:RRDE}, if the initial condition $Z_0$ is any point of the boundary different from the origin, then the solution is unique. Since the origin is the only point where the coefficient in front of the noise vanishes, one may hope that some non-degeneracy of the driving vector fields suffices to recover well-posedness (say in the case of fractional Brownian motion, or more generally if the noise is rough enough in some sense). 
\item {\rm \textbf{One-dimensional driving signals}.}  The driving signal in our equation \eqref{eq:RRDE} is $2$-dimensional. One may conjecture that uniqueness holds for any $1$-dimensional driving signal (namely, in the notation from the introduction, that assertion \textbf{(A)}$_{d,1,p}$ holds for arbitrary $d,p$)\footnote{We thank the anonymous referee for noticing that the case $m=1$ in \textbf{(A)}$_{d,m,p}$ remains open.}. In fact, if $X$ is scalar and the vector field $f$ is never tangent to the boundary of the domain $D$, one may locally transform the equation \eqref{eq:RSDE} into an equation with additive noise (with a similar change of variables as described in Section \ref{sec:CV}), which implies in particular uniqueness of solutions. One may expect that this result still holds for general $f$ and we leave this to further research.
\end{enumerate} 
\end{remark}
}
%
\vspace{4cm}
 \begin{figure}[!h]
\def\svgwidth{255bp}
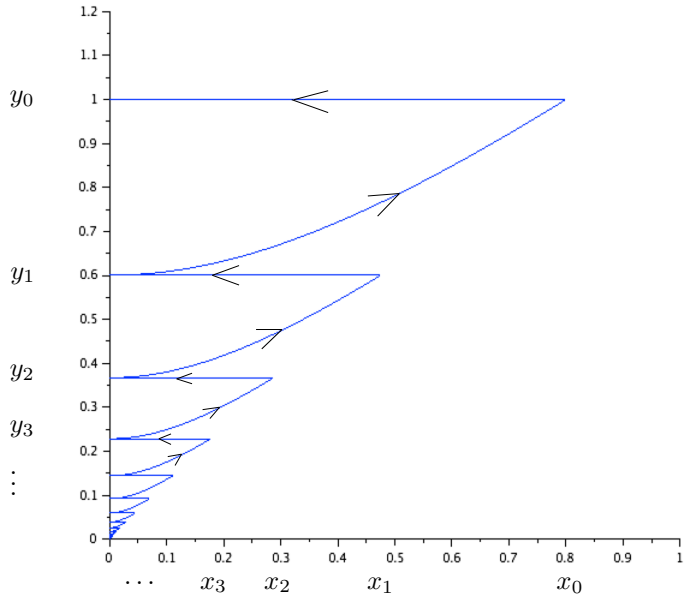
\caption{Trajectory of the solution $Z^+$ obtained in the proof of Theorem \ref{thm:ce1}}
\label{fig:1}
\end{figure}

\begin{figure}[!h]
\includegraphics[width=255bp]{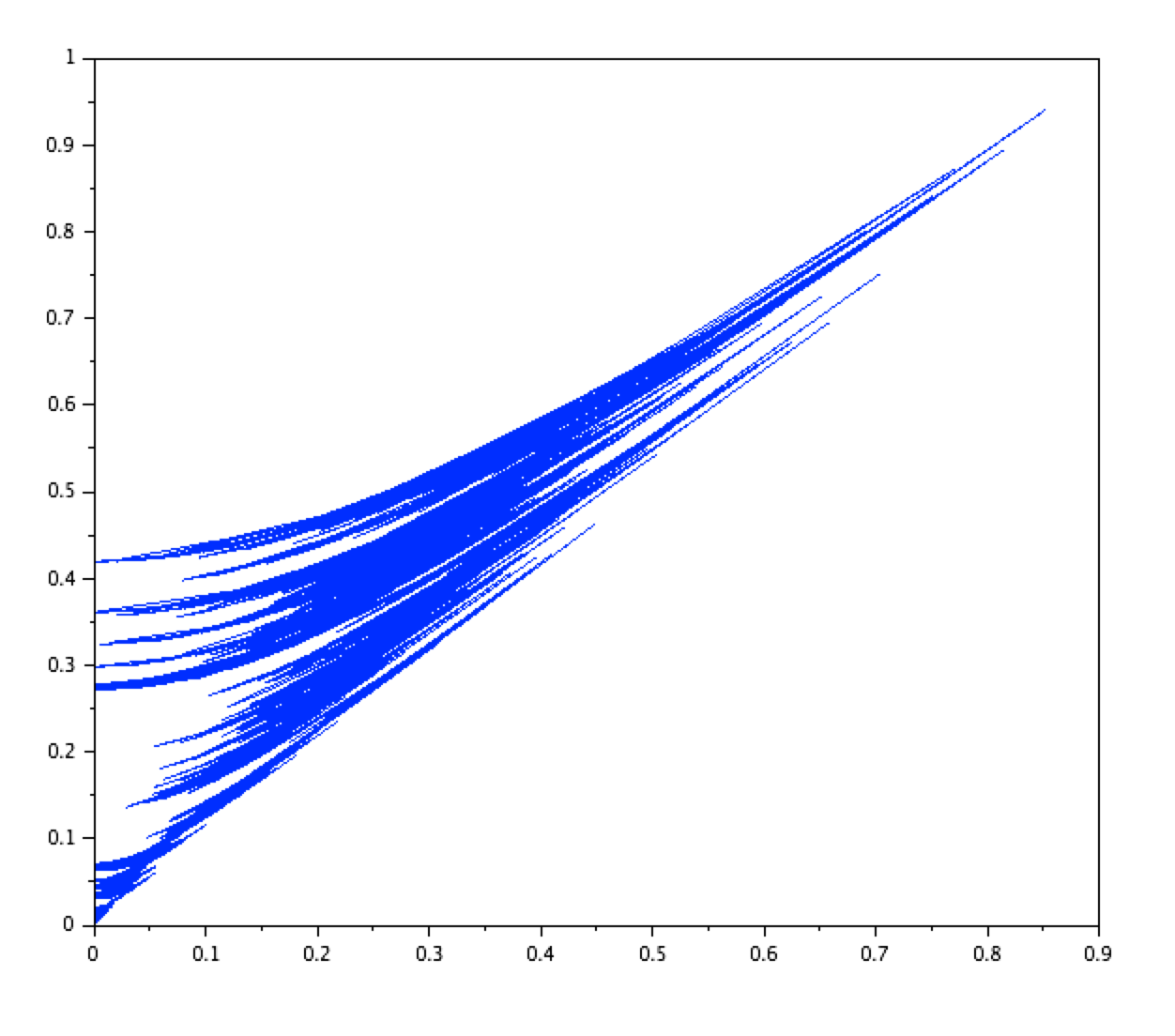}
\caption{Simulation of the trajectory of a solution of \eqref{eq:RRDEdt} where $\lambda$ is a fBm with Hurst index $H=0.2$}
\label{fig:2}
\end{figure}

\newpage

\section{Proofs of the main results}

\subsection{Preliminaries} \label{sec:prel}

\subsubsection{Notation}

Throughout the paper, if $f$ is a function of a real variable, we will denote the value of $f$ at $t$ by either $f_t$ or $f(t)$. We also let $f_{s,t} := f(t) - f(s)$. 

\subsubsection{Definition of solutions}

Let  $\lambda : [0,1]\to \R$ continuous and $\gamma: [0,1] \to \R$ cadlag and nondecreasing. We then define solutions of \eqref{eq:RRDE} on an interval $I \subset [0,1]$ as pairs $(Z,K)$ with  $Z: I \to \R^2$, $K:I \to \R$  such that $K$ is cadlag and nondecreasing, $(Z,K)$ satisfy the second line of \eqref{eq:RRDE}, and for each $s\leq t \in I,$
$$ Z_t = e^{A \lambda_{s,t} }Z_s + \int_{s}^t \left(e^{A \lambda_{u,t}} e_1 \right)(dK_u - d\gamma_u),$$
%
%
%
%

We will also say that $Z$ is a solution if there exists $K$ such that $(Z,K)$ is a solution (in fact, one easily checks that such $K$ is unique but we will not need it). 
\black{We start with a preliminary lemma.}
{\color{black}
\begin{lemma}\label{lem:cont}
Let $Z:[0,1] \to \R^2$ be a solution to \eqref{eq:RRDE} on $(0,1]$ and assume that $Z$ is continuous at $0$. Then $Z$ is a solution to \eqref{eq:RRDE} on $[0,1]$.
\end{lemma}
\begin{proof}
By definition, there exists $K:(0,1] \to \R$ such that $(Z,K)$ is a solution on $(0,1]$, and it suffices to show that $K$ may be extended continuously to $[0,1]$.

Note that for $\delta \in \R$ one has
\begin{equation}
\label{eq:etA}
\exp(\delta A) = \left(\begin{array}{cc} \cosh(\delta)&\sinh(\delta)\\\sinh(\delta)&\cosh(\delta)\end{array}\right)
\end{equation}
and we can then write, for any $s \leq t$ in $(0,1]$,
\begin{align*}
0 \leq K(t) - K(s) \leq \int_s^t \cosh(\lambda_{u,t}) dK_u =   \left(Z_t - e^{A \lambda_{s,t} }Z_s\right)\cdot e_1 + \int_s^t \cosh(\lambda_{u,t}) d\gamma_u.
\end{align*}
Since $Z$ and $\gamma$ are continuous at $0$ this implies that $K(t) - K(s) \to 0$ as $s,t \to 0$ and the result follows.
\end{proof}
}

\subsubsection{A change of variables} \label{sec:CV}

We now introduce a change of variables which will be crucial in the proofs of Theorem \ref{thm:cedt} and Theorem \ref{thm:fbm} (2).

Throughout this section we assume that $\lambda :[0,1] \to \R$ is a fixed continuous path, $\gamma : [0,1] \to \R$ is  nondecreasing. To simplify notation we will also assume in this subsection (w.l.o.g.) that $\gamma$ is continuous as well. Note that this implies that any solution $(Z,K)$ to \eqref{eq:RRDE} must be continuous (indeed, the definition of solution implies that any jumps of $Z$ and $K$ are related by $\Delta Z(t)= \Delta K(t) e_1$, so that any jump of $K$ occurs at a time $t$ with $Z(t) \cdot e_1 >0$, which in turn implies $\Delta K(t) = 0$).

Define 
$$\cR_+ = \left\{ (x,y) , \;\;\; 0 \leq x < y\right\},\;\;\;\;\; \cR_- = \left\{ (x,y) , \;\;\; 0 \leq x < - y\right\},$$
and let 
\begin{align*}
\Psi :  \cR_+ &\to (0,\infty) \times [0,\infty) \\
(x,y) &\mapsto \left(\sqrt{y^2-x^2}, \tanh^{-1}(x/y)\right)
\end{align*}
Then $\Psi$ is a bijection, with inverse given by $\Psi^{-1}(\ell, \delta) = (\ell \sinh(\delta), \ell \cosh(\delta))$. In these new coordinates, \eqref{eq:RRDE} takes the form

\begin{equation}  \label{eq:elldelta}
\left\{ \begin{array}{lll} d \ell  &= &  \sinh(\delta)   d \gamma, \\ d\delta &=& d \lambda - \frac{\cosh(\delta)}{\ell} d \gamma + dk, \\
\delta \geq 0,&& dk \geq 0, \;\;\;\; \delta dk = 0. \end{array}\right.
\end{equation}

Given functions $U =(\ell,\delta) : I \subset [0,1] \to (0,\infty) \times [0,\infty)$, $k : I \to \R$ we say that $(U,k)$ is a solution to \eqref{eq:elldelta} if the third line of \eqref{eq:elldelta} holds and in addition,
\begin{equation}
\label{eq:defEqld}
\forall s \leq t \in I, \;\;\; \ell(t) = \ell(s) + \int_s^t  \sinh(\delta_u)   d \gamma_u, \;\;\;\; \delta(t) = \delta(s) +  \lambda_{s,t} - \int_{s}^t \frac{\cosh(\delta_u)}{\ell_u} d \gamma_u + \int_s^t dk_u.
\end{equation}
We will again say that $U$ is a solution if $(U,k)$ is a solution for some $k$, actually one sees easily that  $U=(\ell,\delta)$ is a solution if and only if
\begin{equation}
\label{eq:defEqldnok}
\forall s \leq t \in I, \;\;\; \ell(t) = \ell(s) + \int_s^t  \sinh(\delta_u)   d \gamma_u, \;\;\;\; \delta(t) = \Gamma_{s} \left( \delta(s) +  \lambda_{s,\cdot} - \int_{s}^{\cdot} \frac{\cosh(\delta_u)}{\ell_u} d \gamma_u \right)(t),
\end{equation}
where $\Gamma_{s}$ is the one-dimensional Skorokhod map defined by
\begin{equation} \label{eq:defGamma1}
\Gamma_{s}(f) : t \mapsto  f(t) - \min\left( \inf_{[s,t]} f, 0\right). 
\end{equation}

\begin{lemma} \label{lem:EllDelta}
Given $U_0 = (\ell_0, \delta_0) \in (0,\infty) \times [0,\infty)$, \eqref{eq:elldelta} admits a unique solution $U=(\ell,\delta)$ on $[0,1]$ starting from $U_0$. In addition,
\begin{equation} \label{eq:EstDelta}
\forall s \leq t \in [0,1], \;\;\;\; \delta_s + \lambda_{s,t} - \gamma_{s,t} \frac{\cosh\left(\delta_s + \sup_{u \in [s,t]} |\lambda_{u,t}|\right)}{\ell_s } \leq \delta_t \leq \delta_s + \sup_{u \in [s,t]} |\lambda_{u,t}|.
\end{equation}
\begin{equation} \label{eq:EstEll}
\forall s \leq t \in [0,1], \;\;\;\; \ell_s \leq  \ell_t \leq \ell_s +  \gamma_{s,t} \sinh\left( \delta_s +  \sup_{u \in [s,t]} |\lambda_{u,t}| \right).
\end{equation}
\end{lemma}

\begin{proof}
1. We first prove that if a solution exists it must satisfy \eqref{eq:EstDelta}-\eqref{eq:EstEll}. Monotonicity of $\ell$ is obvious from non-negativity of $\delta$. The upper bound in \eqref{eq:EstDelta} is obtained by considering $u =  \sup \left\{ r \in [s,t], \delta(r) = 0\right\}$. The lower bound follows from the upper bound and monotonicity of $\ell$. Finally, the upper bound in \eqref{eq:EstEll} follows from the upper bound in \eqref{eq:EstDelta}.

2. We prove existence. In the case when $\lambda$ is smooth existence is classical. We then take approximations $\lambda^n \to \lambda$ in supremum norm, and let $(\delta^n, \ell^n)$ be the corresponding solutions. By Step 1., one has a uniform upper bound on $\delta^n$,  $\ell^n$ and $(\ell^n, \delta^n)$ are equicontinuous. By Arzela-Ascoli, we may find a subsequence which converges uniformly to some $(\ell,\delta)$ and it is clear that the definition of solution is stable under passage to the limit in supremum norm.

3. We then prove uniqueness. This is straightforward due to additivity of the noise (cf e.g. \cite{CMR17} for a similar proof). If $(U,k)$ and $(U',k')$ are solutions then noting that $V(\ell,\delta) := (\sinh(\delta), - \cosh(\delta)/\ell)$ is locally Lipschitz on $(0,\infty) \times [0,\infty)$, one has that
\begin{align*}
d  \left| U - U'\right|^2& =  2 (U-U') \cdot (V(U) - V(U')) d \gamma_t +  2 (U -U') \cdot (dk - dk') e_1 \\
&\leq C \left| U - U'\right|^2 d \gamma_t 
\end{align*}
since $ (U -U') \cdot (dk - dk') e_1 = - \ell dk' - \ell' dk \leq 0$, and we conclude by Gronwall's lemma.
\end{proof}

\begin{lemma} \label{lem:ZEll}
(1) $Z$ is a solution to \eqref{eq:RRDE} with values in $\cR_+$ if and only if $\Psi(Z)$ is a solution to \eqref{eq:elldelta}, and then $dK =\frac{dk}{\ell}$.

(2) For any $Z_0$ in $\cR_+$, there exists a unique solution to \eqref{eq:RRDEdt} on $[0,1]$ starting from $Z_0$. In addition, $Z_t \in \cR_+$ for each $t \in [0,1]$. The same holds if $\cR_+$ is replaced by $\cR_-$.

(3) Let $Z$ be a solution to \eqref{eq:RRDEdt} starting from $0$, and let $t_{\ast} = \inf\{ t , Z_t \neq 0\}$. Then either $Z_s \in \cR_+$ for each $s>t_{\ast}$, or $Z_s \in \cR_-$ for each $s>t_{\ast}$.

(4) If $Z$ is a solution to \eqref{eq:RRDEdt} starting from $0$, then $(\ell,\delta):= \Psi^{-1}(Z)$ is a solution to \eqref{eq:elldelta} on $(t_\ast,1]$, and one has $\lim_{s \downarrow t_\ast} (\ell_s, \delta_s)= (0,0)$, and $\delta_{s_n} = 0$ for some sequence $s_n \to t_\ast$.

(5) Let $\mathcal{Z} = \{0\} \cup \Psi^{-1}(K)$ where $K$ is a bounded subset of $(0,\infty) \times \R_+$. Then
$$\left\{ Z \mbox{ solution to \eqref{eq:RRDE} on }[0,1] , \;\;\; Z(0) \in \mathcal{Z} \right\}$$
is precompact in $C([0,1])$.
\end{lemma}

\begin{proof}
(1) : 
Fix $(Z,K)$ a solution to \eqref{eq:RRDE} with values in $\cR_+$, and let $\lambda^n \to \lambda$ be smooth approximations. We let 
$$Z^n(t) = e^{A \lambda^n_{0,t}} Z(0) - \int_0^t e^{A \lambda_{u,t}^n} e_1 \left( d\gamma_u - dK_u \right).$$
Then a simple calculus exercise shows that $U^n  = (\ell^n, \delta^n) := \Psi(Z^n)$ solves
\begin{equation*}
d\ell^n = \sinh(\delta^n) (d \gamma - dK), \;\;\; d \delta^n = d \lambda^n + \frac{\cosh(\delta^n)}{\ell^n} \left(d \gamma - dK\right)
\end{equation*}
Note that $Z^n \to Z$ in supremum norm, so that when $n \to \infty$,  $\sinh(\delta^n) dK \to 0$, $\frac{\cosh(\delta^n)}{\ell^n} dK \to \frac{dK}{\ell}$. This implies that $\Psi(Z) = \lim_{n} \Psi(Z^n)$ solves \eqref{eq:elldelta}. The converse implication can be proven similarly.

{\color{black}
Given $Z_0 \in \cR_+$, Lemma \ref{lem:EllDelta} and (1) imply that there exists a unique solution $Z$ to \eqref{eq:RRDE} such that $Z_t \in \cR_+$, $t\in[0,1]$. By continuity of solutions this is actually the only solution to \eqref{eq:RRDE} (if $\hat{Z}$ is any other solution, then $\hat{Z}$ coincides with $Z$ up to $\hat{t}=\inf\{ t : \hat{Z}(t) \notin\cR_+\}$, but then if $\hat{t}$ is finite, $\hat{Z}(\hat{t})=Z(\hat{t}) \in \cR_+$ which is not possible since $\cR_+$ is relatively open in $\R_+ \times \R$).  The case when $Z_0 \in \cR_-$ is similar by symmetry. This proves (2).}

(2) clearly implies that if $Z$ is a solution and $Z(t) \in \cR_+$ for some $t$, then $Z(s) \in \cR_+$ for all $s \geq t$ (and idem for $\cR_-$). A similar analysis shows that in the region $\{ x > |y|\}$, $x^2 - y^2$ must be non-increasing in $t$, so that this region is not attainable from $0$. This proves (3).

Ad (4), since $Z_{t_\ast} = 0$ it is clear that $\ell(s)$ converges to $0$ as $s \downarrow t_{\ast}$. It is also clear that one must have a sequence $s_n \to t_\ast$ with $\delta(s_n)=0$ (otherwise, $dK \equiv 0$ on a neighborhood of $t_{\ast}$ which implies that $Z \equiv 0$ as well). By (1) and \eqref{eq:EstDelta}, this in turn implies that $\delta$ converges to $0$ at $t_{\ast}$.

(5) is a consequence of points (1), (4) and \eqref{eq:EstDelta}-\eqref{eq:EstEll}.
\end{proof}

\begin{remark} \label{rmk:noGamma}
In the case when $\gamma \equiv 0$, the equation simplifies and we have that for an initial condition $Z \in \cR_{+}$, the unique solution is simply given by $Z_t = \Psi^{-1}(\ell_t, \delta_t)$ with
$$\ell_t = \ell_0, \;\;\; \delta_t = \Gamma_{0}\left( \delta_0 + \lambda_{0,\cdot} \right)(t).$$
\end{remark}

\subsection{Proof of Theorem \ref{thm:ce1} and Theorem \ref{thm:fbm} (1)}

\begin{proof}[Proof of Theorem \ref{thm:ce1}]
Let $(\delta_k)_{k \geq 0}$ be a sequence of non-negative numbers converging to $0$ and such that
\begin{equation} \label{eq:condInf2var}
\sum_{k=0}^\infty \delta_k^2 = +\infty,
\end{equation}
\begin{equation} \label{eq:condGamBV}
\sum_{k=0}^\infty \delta_k \exp\left( - \frac{1}{2}\sum_{0\leq j \leq k} \delta_j^2\right) < +\infty,
\end{equation}
and which further satisfies 
\begin{equation} \label{eq:pvar}
\forall p > 2, \;\;\;\sum_{k=0}^\infty \delta_k^p < +\infty,
\end{equation}
(for instance taking $\delta_k \sim C k^{-1/2}$ with $C>1$ will do).

Fix $(t_{k})_{k\geq 0}$ a decreasing sequence with $t_0=1$ and $t_k \to 0$ as $k \to \infty$. We then define $\lambda :[0,1] \to \R $ continuous, affine on each $(t_{k+1},t_{k})$, $k\geq 0$, with increments
\begin{equation}
\lambda(t_{3k+2}) - \lambda(t_{3k+3}) = - \delta_k, \;\;\; \lambda(t_{3k+1}) - \lambda(t_{3k+2}) =  \delta_k, \;\;\;\;  \lambda(t_{3k}) - \lambda(t_{3k+1}) = 0, \;\;\;\;k \geq 0,
\end{equation}
and note that $\lambda$ has infinite $2$-variation but finite $p$-variation for each $p>2$.

We then define recursively $y_k$, $k \geq 0$ by 
\begin{equation}
y_0 =1, \;\;\; y_{k+1} = y_k / \cosh(\delta_k)
\end{equation}
and note that 
\black{ for $k \geq 1$
\begin{equation} \label{eq:logy}
\ln(y_k) = - \sum_{i=0}^{k-1} \ln(\cosh(\delta_i)) \leq C - \frac{1}{2}  \sum_{i=0}^{k}  \delta_i^2
\end{equation}
for some finite constant $C$, where we have used that $\ln(\cosh(\delta)) = \frac{\delta^2}{2} + O(\delta^4)$ as $\delta \to 0$ as well as \eqref{eq:pvar}. Using  \eqref{eq:condInf2var}, this implies that $y_k \to 0$ as $k \to \infty$.}
We further define
$$x_k = \sinh(\delta_k) y_{k+1}$$ 
and let $\gamma$ be again continuous affine on each $(t_k,t_{k+1})$, with
\begin{equation}
\gamma(t_{3k+2}) - \gamma(t_{3k+3}) = \gamma(t_{3k+1}) - \gamma(t_{3k+2}) = 0 , \;\;\;\gamma(t_{3k}) - \gamma(t_{3k+1}) =  x_k , \;\;\;\;k \geq 1.
\end{equation}
\black{Note that for $k$ large enough, $x_k \leq 2 \delta_k y_{k+1}$ so that, by \eqref{eq:condGamBV} and  \eqref{eq:logy}, $\gamma$ has finite variation on [0,1].}

We now exhibit a non-null solution to \eqref{eq:RRDE}. Let $Z^{(k)}$ be the solution to \eqref{eq:RRDE} but starting at time $t_{3k}$ from the point $(0,y_k)$. Then one can show inductively that for $j<k$, it holds that
$$Z^{(k)}(t_{3j+2}) = Z^{(k)}(t_{3j+3}) = (0,y_{j+1}), \;\; Z^{(k)}(t_{3j+1}) = (x_j,y_j), \;\;\; Z^{(k)}(t_{3j}) = (0,y_j).  $$
{\color{black}
Indeed, on intervals of the form $[t_{3j+3}, t_{3j+2}]$, the term $AZ d\lambda$ pushes in the outer normal direction, namely  for $t$ in this interval one has
\[ Z(t) = (0,y_{j+1}), \;\; K(t) - K(t_{3j+3}) = - y_{j+1} \lambda_{t_{3j+3}, t}, \]
whereas for $t$ in $[t_{3j+2}, t_{3j+1}]$, recalling the expression \eqref{eq:etA}, one has
\[ Z(t) = (\sinh(\lambda_{3j+2,t}) y_{j+1}, \cosh(\lambda_{3j+2,t}) y_{j+1}), \;\;\;\; K(t) = K(t_{3j+2}), \]
and finally on intervals of the form $[t_{3j+1}, t_{3j}]$ the drift $\gamma$ simply pushes back $Z$ to the $y$-axis, namely
\[Z(t) = (\sinh(\delta_j) y_{j+1} - \gamma_{t_{3j+1}, t}, \cosh( \delta_j) y_{j+1}), \;\;\;\; K(t) = K(t_{3j+1}). \]
}


Then let $Z^{+}(t) = \lim_{k} Z^{(k)}(t)$ if $t>0$ (actually the sequence is constant for $k$ large enough) and $Z^{(+)}(0)=0$. Then $Z^{(+)}$ is a solution to \eqref{eq:RRDE}. Indeed, it is a solution on each $(\epsilon,1]$ for $\epsilon >0$ and by Lemma \ref{lem:cont} this implies that it is a solution on $[0,1]$. 
See Figure 1 for a picture of the trajectory of $Z^{+}$.

Finally, for each $\eta \in [0,1]$, taking $y'_k = \eta y_k, x'_k=\eta x_k$, one may construct in exactly the same way a solution $Z^\eta$ to \eqref{eq:RRDE} which satisfies $Z^\eta(t_{3k})=(0,y'_k)$ for all $k \geq 0$. Indeed, $(y'_k)$ satisfies the same induction relation as $(y_k)$, so that on each interval $[t_{3k+3},t_{3k}]$, the first two steps are exactly the same, and the third step also brings $(x'_k,y'_k)$ to $(0,y'_k)$ since $x'_k \leq x_k$. It follows that \eqref{eq:RRDE} has uncountably many different solutions.
\end{proof}

\black{
\begin{remark}
Note that, in the dynamics described above, $Z$ is constant on intervals of the form $[t_{3j+3},t_{3j+2}]$ (where $Z$ is on the $y$ axis while $\lambda$ decreases), and one may then wonder if these intervals could not simply be dropped from the definition of solutions. However this is not possible, since in order to keep $\lambda$ bounded near $0$ (as it should be, since it must be a continuous function), it is needed to have intervals where $\lambda$ decreases in order to compensate for its infinite variation on the intervals where it increases. The effect of the reflection is to allow for such intervals without any impact on the dynamics of the solution.\end{remark}
}

We note that the counterexample can actually be applied to more general paths $\lambda$ (allowing the drift $\gamma$ to contain jumps), the proof of the below proposition is exactly as above.

\begin{proposition} \label{prop:1}
Let $\lambda \in C([0,1],\R)$ be such that, for some sequence of times $1=t_0 \geq t_1 \geq \ldots t_k \downarrow 0$, letting 
$$\delta_k = \lambda(t_k) - \min_{[t_{k+1},t_{k}]} \lambda,$$
one has that \eqref{eq:condInf2var}-\eqref{eq:condGamBV} hold. Then, for any summable sequence $(x_k)_{k \geq 0}$ with 
$$\forall k \geq 0, \;\;\;\; x_k \geq \delta_k \Pi_{j=0}^k \cosh(\delta_j)^{-1},$$ letting $\gamma$ be the piecewise constant path defined by
\begin{equation} \label{eq:DefGammaJump}
\gamma(t) =  \sum_{0< t_k \leq t} x_k,
\end{equation}
the equation \eqref{eq:RRDE} has infinitely many solutions. 
\end{proposition}
%
%

One may wonder what is the minimal regularity of paths $\lambda$ to which the above proposition applies. Clearly  they must have infinite $2$-variation, but the lemma below shows that they must have at least infinite $\psi$-variation, with $\psi(r) = r^2/\log(1/r)$, which in particular rules out classical Brownian paths.

\begin{lemma} \label{lem:logvar}
Let $(\delta_k)_{k \geq 0}$ be a sequence of elements of $(0,1)$, such that
$$\sum_{k \geq 0} \frac{\delta_k^2} {\log (\delta_k^{-1})} < \infty \mbox{ and } \sum_{k \geq 0} \delta_k \exp\left( - \frac{1}{2} \sum_{0 \leq j\leq k} \delta_j^2 \right) < \infty.$$

Then
$$\sum_{k \geq 0} \delta_k^2 < \infty.$$ 
\end{lemma}

\begin{proof}
Re-labelling if necessary we may assume that $(\delta_k)$ is non-increasing in $k$. Let $k_0$ be such that $\sum_{k \geq k_0} \frac{\delta_k^2} {\log (\delta_k^{-1})} \leq \frac{1}{2}$. We then write
\begin{align*}
\sum_{k \geq k_0} \delta_k^{2} &= \sum_{ k \geq k_0} \delta_k \exp\left(-  \log(\delta_k^{-1})\right) \\
&\leq \sum_{ k \geq k_0} \delta_k \exp\left(- \sum_{k_0 \leq j < k } \frac{ \delta_j^2  \log(\delta_k^{-1})}{\log(\delta_j^{-1})}\right) \\
&\leq \sum_{ k \geq k_0} \delta_k \exp\left(-\frac{1}{2} \sum_{k_0 \leq j < k } \delta_j^2 \right)  < \infty.
\end{align*}
\end{proof}

We will then prove Theorem \ref{thm:fbm} (1). We will not apply directly Proposition \ref{prop:1} above, since we want our solutions to be adapted processes, whereas the construction above is anticipative (note that $Z(t_k)$ depends on $(\delta_j)_{0\leq j \leq k}$).

\begin{proof}[Proof of Theorem \ref{thm:fbm} (1)]
Fix $0< H < \frac{1}{2}$. Let $t_0=1$ and $t_k$ decreasing to $0$ be such that $t_{k} - t_{k+1} \sim k^{-\alpha}$ for some $\alpha$ satisfying
\begin{equation*}
\max\left(\frac{5}{12H} ,1 \right)< \alpha < \frac{1}{2H},
\end{equation*}
and let $(x_k)_{k\geq 0}$ be a summable sequence with
\begin{equation} \label{eq:defxn}
x_k \geq \exp\left(-k^{\theta}\right)
\end{equation}
for some $0<\theta<1- 2 \alpha H$. We then let $\gamma$ be defined by \eqref{eq:DefGammaJump}.

We further define
$$\delta_k = \lambda(t_k) - \min_{[t_{k+1},t_k]} \lambda, \;\; k\geq 0$$
$$S_N  = \sum_{k=0}^N \delta_k^2, \;\;\;\;\tilde{S}_N = S_N - \E[S_N]$$
and by Gaussian computations (deferred to subsection \ref{sec:ProofSN}), it holds that 
\begin{equation} \label{eq:asSN}
\P^H-\mbox{a.s.}, \lim_{N\to \infty} \tilde{S}_N  = \tilde{S}, \mbox{ for some finite r.v. }\tilde{S}.
\end{equation}

We then fix $\eta >0$ and let $Z^{N,\eta}$ be the solution to \eqref{eq:RRDE} starting at time $t_{N}$ from $(0, y^{N, \eta}_{N})$ with 
$$y^{N,\eta}_{N} = \eta \exp\left(- \frac{1}{2}\E[S_{N-1}]\right)$$

Define for $0 \leq k \leq N$
$$y^{N, \eta}_k = y^{N,\eta}_{N} \Pi_{j=k}^{N-1} \cosh(\delta_j).$$
Then it holds that for $0 < k \leq N$,
\begin{equation} \label{eq:ZNinduction}
Z^{N,\eta}(t_{k}) = (0, y^{N, \eta}_k), \;\;\; x_k \geq \sinh(\delta_{k-1})y^{N, \eta}_k  \;\;\;\;\;\;\Rightarrow \;\;\;\;\;\;  Z^{N,\eta}(t_{k-1}) = (0, y^{N, \eta}_{k-1}) .
\end{equation}
{\color{black}
Indeed, by Remark \ref{rmk:noGamma}, if the first equality holds, then one has in that case that for $t_k \leq t < t_{k-1}$
\[ Z^{N,\eta}(t) = (\sinh(\delta(t)) y^{N, \eta}_k, \cosh(\delta(t)) y^{N, \eta}_k) \]
with $\delta(t) = \Gamma_{t_k}\left( \lambda_{t_k,\cdot} \right)(t)$, and in particular $\delta(t_{k-1}-) = \delta_{k-1}$. Then if the second inequality above holds, the jump of $\gamma$ at time $t_{k-1}$ brings $Z$ back to the axis, so that  $Z^{N,\eta}(t_{k-1}) = (0, \cosh(\delta_{k-1})y^{N, \eta}_{k})$.

Now note that
\begin{equation} \label{eq:esty}
y^{N, \eta}_k = \eta \exp\left(-\frac{1}{2} \E \left[S_{k-1}\right]\right) \exp\left( - \frac{1}{2} \left ( \tilde{S}_{N-1} - \tilde{S}_{k-1} \right) + R_{k,N}\right),
\end{equation}
where
\[ R_{k,N} = O\left( \sum_{j=k}^\infty \delta_j^4 \right).\]
By the scaling properties of fBm, it holds that $\E[\delta_j^4] = O (j^{-4H \alpha})$, so that a.s., $R_{k,N} = o_{k\to +\infty}(1)$.

In addition, for some constant $C>0$
$$\E \left[S_{k-1}\right] \sim_{k \to \infty} C k^{1-2\alpha H},$$
so that by \eqref{eq:defxn},  \eqref{eq:asSN}, and the fact that $\delta_k \to 0$ a.s., we can use \eqref{eq:ZNinduction} to obtain that there exists $\P^H$-a.s. $k_0$ s.t. for each $N \geq k \geq k_0$,
$$Z^{N,\eta}(t_k) = (0, y^{N, \eta}_k).$$
}
In addition for $k_0 \leq k \leq N \leq M$, it holds that
$$\frac{y^{N, \eta}_k}{y^{M, \eta}_k}  = \exp\left(\frac 1 2 \left( \tilde{S}_{N} - \tilde{S}_{M}\right) + o_{N,M \to \infty}(1) \right)\to_{N,M \to \infty} 1$$
 and it follows that $Z^{N,\eta}(t_k)$ converges to some limit $Z^\eta(t_k)$ for each $k \geq k_0$, which is non-null by \eqref{eq:esty}. Also note that by Remark \ref{rmk:noGamma}, there exists a continuous map $\psi :\R^2 \times \R \times \R \to \R^2$ such that $Z$ is a solution to \eqref{eq:RRDE} on $[t_j,1]$ if and only if
 $$\forall t_j \leq t_{k+1} \leq t < t_k \leq 1, Z(t) = \psi\left( Z(t_{k+1}), \lambda(t), \min_{[t_{k+1},t]} \lambda\right),$$
 $$\forall t_{k} > t_j, Z(t_k) = \Pi \left( Z(t_k-) - x_k e_1\right),$$
 where $ \Pi(x,y) = (\max(x,0),y)$. In particular, it follows that $Z^{N,\eta}$ converges a.s. for each $t \in (0,1]$ as $N \to \infty$ to a solution $Z^\eta$ of \eqref{eq:RRDE} on $(0,1]$. In addition, from \eqref{eq:esty} it is also clear that $Z^\eta(t_k) \to 0$ as $k \to \infty$, so that, by Lemma \ref{lem:cont}, $Z^\eta$ may be extended to a solution on $[0,1]$ with $Z_0=0$.

\end{proof}

\subsection{ Proof of Theorem \ref{thm:cedt} and Theorem \ref{thm:fbm} (2)}~\\

Let us first give a sketch of the argument, which is based on the change of variables described in subsection \ref{sec:CV}, where we obtained the equivalent equation \eqref{eq:elldelta}. To simplify, replace in this equation $\sinh(\delta)$ by $\delta$ and $\cosh(\delta)$ by $1$. One may then rewrite it (assuming $\ell_0=\delta_0=0$)
as an equation involving only $\ell$, namely
\begin{equation} \label{eq:ell}
\ell_t = \int_0^t \Gamma_0 \left(\lambda_{0,\cdot} - \int_0^\cdot \frac{du}{\ell_u}
\right)(s) ds,
\end{equation}
(recall that $\Gamma_{0}$ is the one-dimensional Skorokhod map defined in \eqref{eq:defGamma1}). Therefore the existence of a non-zero solution to \eqref{eq:RRDEdt} essentially coincides with the existence of a non-zero solution to \eqref{eq:ell}. Note that if $\lambda$ admits $\omega$ as a modulus then $\Gamma_0 \left(\lambda_{0,\cdot} - \int_0^\cdot \frac{du}{\ell_u}\right)(t) \leq \theta_\omega(\ell(t))$, so that such a solution satisfies $\frac{d}{dt} \ell(t) \leq \theta_{\omega}(\ell(t))$. Under Osgood's condition on $\theta_\omega$, this implies that $\ell$ must actually be identically zero.  When Osgood's condition does not hold, there exist solutions to $\frac{d}{dt}\ell=\theta_\omega(\ell)$ which escape from zero, and by a suitable discretization we are also able to exhibit $\lambda \in \C_\omega$ such that the solution to \eqref{eq:ell} has a similar behaviour. Finally, in the case of a fractional Brownian motion, we use a similar discretization argument combined with small ball properties for the reflected fBm, to obtain Theorem \ref{thm:fbm} (2).

Let us now proceed with the rigorous proof. 

\begin{proof}[Proof of Theorem \ref{thm:cedt}]

Let us first treat the case where Osgood's condition holds. Asume by contradiction that we have a solution $Z$ to \eqref{eq:RRDEdt} with $Z_0=0$ and $Z \neq 0$ on $(0,1]$. Let $(\ell,\delta)$ be the corresponding solution to \eqref{eq:elldelta} obtained by Lemma \ref{lem:ZEll}. Note that by Lemma \ref{lem:EllDelta}, $\delta$ is bounded and one has a constant $C$ (depending on $\lambda$) s.t. $\sinh(\delta_t) \leq C \delta_t$ for all $t \in [0,1]$.

Let $s >0$ be such that $\delta(s) = 0$. Then $\ell$ is differentiable in $t$ for $t>s$, with
\begin{align*}
\frac{d}{dt}(\ell(t)) \leq  \;& C \Gamma_{s} \left( \lambda_{s,\cdot} - \int_s^{\cdot} \frac{\cosh(\delta_u)}{\ell_u} du \right)(t) \\
\leq\; & C  \Gamma_{s} \left(  \lambda_{s,\cdot} - - \frac{1}{  \ell_t} (\cdot - s) \right)(t) \\
 \leq\; & {\color{black}C \sup_{u \in [s,t]} \left( \omega(t-u) - \frac{t-u}{\ell_t} \right)}\\
\leq \; & C \theta_\omega( \ell(t)). 
\end{align*}
where in the second inequality, we have used that $\Gamma_{s}(f) \leq \Gamma_{s}(g)$ if $(g-f)$ is nondecreasing and $(g-f)(s)=0$. This implies that for each $0< s \leq t$ with $\delta(s)=0$, one has
$$t \geq s + \int_{\ell(s)}^{\ell(t)} \frac{du}{C \theta_{\omega}(u)}.$$
On the other hand, by Lemma \ref{lem:ZEll}, there exists $s_n \to 0$ with $\delta(s_n) = 0$. We therefore obtain that
$$\forall t >0, \;\; t \geq \int_{0}^{\ell(t)} \frac{du}{C\theta_{\omega}( u)} = +\infty,$$
a contradiction.


We then treat the case where Osgood's condition does not hold, and want to find a nonzero solution to \eqref{eq:elldelta} starting from $0$. 

{\color{black}
We first claim that for some $K>0$ it holds that, for all $\varepsilon>0$ small enough,
\begin{equation} \label{eq:ineqthetaint}
 \int_0^{\varepsilon \theta_\omega(\varepsilon)} \left(\omega(r) - \frac{2r}{\varepsilon} \right)_+ dr \geq K^{-1} \varepsilon \theta_\omega(\varepsilon)^2.
\end{equation}

Recall that we assume that for some $\kappa>0$ it holds that
\begin{equation*} \label{eq:asnom}
\liminf_{\eta \to 0} \frac{2 \omega(\eta)}{\omega(2\eta)} > 1 + 2 \kappa,
\end{equation*}
and let $\rho(\varepsilon):= \sup \{\eta >0, \varepsilon \omega(\eta) \geq \eta \}$, then (for $\varepsilon$ small enough) 
$$\varepsilon \theta_\omega(\varepsilon) \geq \varepsilon\omega(\rho(\varepsilon)/2) - \frac{\rho(\varepsilon)}{2} \geq {\kappa} \varepsilon \omega(\rho(\varepsilon))  = \kappa \rho(\varepsilon),$$
and in particular 
$$\theta_\omega(\varepsilon) \leq \omega(\rho(\varepsilon)) \leq \omega\left(\kappa^{-1} \varepsilon \theta_\omega(\varepsilon) \right) \leq M \omega\left( \varepsilon \theta_\omega(\varepsilon) \right)$$
(where $M$ is any integer greater than $\kappa^{-1}$). Then noting that if $s \in [r/2, r]$ for $r\geq 0$ one has $\omega(s) - \frac{2s}{\varepsilon} \geq \frac{\omega(r)}{2} - \frac{2r}{\varepsilon}$, we obtain 
$$ \int_0^{\varepsilon \theta_\omega(\varepsilon)} \left(\omega(r) - \frac{2r}{\varepsilon} \right)_+ dr \geq \sup_{ r \in [0, \varepsilon \theta_\omega(\varepsilon)]} \frac{r}{4} \left(\omega(r) - \frac{4r}{\varepsilon}\right).$$
Let $n \geq 1$ be fixed such that
$$(1+2 \kappa)^n \geq 5 M,$$
then one has
\begin{align*} \int_0^{\varepsilon \theta_\omega(\varepsilon)} \left(\omega(r) - \frac{2r}{\varepsilon} \right)_+ dr& \geq 2^{-2-n} \varepsilon \theta_\omega(\varepsilon) \left( \omega(2^{-n} \varepsilon \theta_\omega(\varepsilon)) - 2^{2-n} \theta_\omega(\varepsilon) \right) \\
&\geq 2^{-2-2n} \varepsilon \theta_\omega(\varepsilon) \left( (1+2 \kappa)^n M^{-1} \theta_\omega(\varepsilon) - 4 \theta_\omega(\varepsilon) \right) \\
&\geq 2^{-2-2n} \varepsilon \theta_\omega(\varepsilon)^2,
\end{align*}
which concludes the proof of \eqref{eq:ineqthetaint}.
}


We will then follow a discretization procedure. We start from  $\varepsilon_0=1$, and given $\varepsilon_k$, 
there exists a unique $\varepsilon_{k+1}$ such that

\begin{equation}
\varepsilon_k = \varepsilon_{k+1} + \frac{1}{2 K} \theta_\omega(\varepsilon_k)  \theta_\omega\left(\varepsilon_{k+1}\right) \varepsilon_{k+1}.
\end{equation}
{\color{black}
Then $\varepsilon_k$ decreases to a limit $\underline{\varepsilon}$ as $k \to \infty$, with $\underline{\varepsilon} = \underline{\varepsilon} + \frac{1}{2 K} \theta_\omega(\underline{\varepsilon})^2 \underline{\varepsilon}$, so that, since $\theta_\omega$ is strictly positive on $(0,+\infty)$, it holds that actually
\[\lim_{k \to \infty} \varepsilon_k = 0. \]
In addition, since $\theta_\omega(\varepsilon_k) \to 0$ as $k \to \infty$ one has that 
\begin{equation} \label{eq:equivTheta}
 \theta_\omega(\varepsilon_{k+1}) \sim \theta_\omega(\varepsilon_{k}) \mbox{ as } k \to \infty.\
 \end{equation}
 }
 

We let $\eta_{k+1} := \theta_\omega\left(\varepsilon_{k+1}\right) \varepsilon_{k+1}$ and
note that since $\theta_\omega$ is nondecreasing, one has
\begin{equation}
\int_{\varepsilon_{k+1}}^{\varepsilon_k} \frac{ds}{\theta_\omega(s)} \geq \frac{1}{2 K}  \eta_{k+1},
\end{equation}
which clearly implies that $\sum_{k\geq 0} \eta_k < \infty$. Changing $\varepsilon_0$ if necessary we may assume that this sum is less than $\frac 1 2$.
Then we let $t_k$ converging to $0$ with $t_{k} - t_{k+1} = 2 \eta_{k+1}$, and let $\lambda$ be defined by :
\[
\begin{array}{cccc}
 \lambda(t_{k+1} + r) &=& \omega(r), &0 \leq r \leq \eta_{k+1},  \\
 \lambda(t_{k+1} + \eta_{k+1} + r) &= &\omega(\eta_{k+1} - r),  &0 \leq r \leq \eta_{k+1}. 
 \end{array}
\]
Then $\lambda$ is in $\C_\omega$.

We note that for any solution of \eqref{eq:elldelta} with $\delta_0=0$, one has $\cosh(\delta(t)) \leq C(t)$ on $[0,1]$ for some $C(t) \to 1$ as $t \to 0$. Considering a smaller interval if necessary we may assume that $C(\cdot) \leq 2$. 

We then note that if $(\ell,\delta)$ is a solution to \eqref{eq:elldelta} with $\delta(0)=0$ such that $\ell(t_{k+1}) \geq \varepsilon_{k+1}$, then
{\color{black}
\begin{align*}
\ell(t_k) \geq \ell(t_{k+1} + \eta_{k+1}) &= \ell(t_{k+1}) + \int_{t_{k+1}}^{t_{k+1} + \eta_{k+1}} \sinh(\delta(s)) ds \\
&\geq \varepsilon_{k+1} + \int_{t_{k+1}}^{t_{k+1} + \eta_{k+1}} \delta(s) ds
\end{align*}
and since for $s \in [t_{k+1},t_{k+1} + \eta_{k+1}]$ one has
\[
\delta(s) \geq \delta(t_{k+1}) + \lambda_{t_{k+1},s} - \int_{t_{k+1}}^s \frac{\cosh(\delta(u))}{\ell(u)} du \geq \omega(s-t_{k+1}) - \frac{2 (s-t_{k+1})}{\varepsilon_{k+1}},
\]
we obtain
}
\begin{align*}
\ell(t_k) &\geq \varepsilon_{k+1} + \int_0^{\eta_{k+1}} \left( \omega(r) - \frac{ 2 r}{\varepsilon_{k+1}}\right)_+ dr \\
& \geq  \varepsilon_{k+1} + \frac{1}{K} \eta_{k+1} \theta_\omega(\varepsilon_{k+1}) \\
&\geq \varepsilon_{k+1} + \frac{1}{2 K} \eta_{k+1} \theta_\omega(\varepsilon_{k}) \\
&= \varepsilon_k,
\end{align*}
at least for $k$ large enough, where we have used \eqref{eq:ineqthetaint} in the second inequality and \eqref{eq:equivTheta} in the third one.

{\color{black}
We have therefore proven the existence of $k_0$ such that for any solution $(\ell,\delta)$ to \eqref{eq:elldelta} with $\delta(0)=0$, it holds that 
\begin{equation} \label{eq:k0}
\mbox{if } \ell(t_k) \geq \varepsilon_k \mbox{ for some }k \geq k_0, \mbox{ then  }\ell(t_j) \geq \varepsilon_j, \;\mbox{ for all }k_0 \leq j \leq k .
\end{equation}
}
For $\gamma>0$ we then let $Z^\gamma$ be the unique solution to \eqref{eq:RRDEdt} starting from $(0,\gamma)$. By Lemma \ref{lem:ZEll} (5), we can find a subsequence $\gamma' \to 0$ with $Z^{\gamma'}$ converging to a solution $Z$ of \eqref{eq:RRDEdt} starting from $(0,0)$. 
{\color{black}
Let $(\delta^{\gamma'},\ell^{\gamma'}) = \Psi(Z^{\gamma'})$, where $\Psi$ is defined in section \ref{sec:CV}. Then $(\delta^{\gamma'},\ell^{\gamma'})$ is a solution to \eqref{eq:elldelta} with $(\delta^{\gamma'},\ell^{\gamma'}) = (0,\gamma')$. Since $\varepsilon_k \to 0$, there exists $k_{\gamma'}$ such that $\varepsilon_k \leq \gamma'$ for all $k \geq k_{\gamma'}$, and then for each $k \geq k_{\gamma'}$, one has that $\ell^{\gamma'}(t_k) \geq \ell^{\gamma'}(0) \geq \varepsilon_k.$ Applying \eqref{eq:k0}, we obtain that
\[\forall \gamma', \forall k\geq k_0, \;\;\; \ell^{\gamma'}(t_k) \geq \varepsilon_k\]
and letting $\gamma \to 0$ we obtain that $Z(t_k) \neq 0$ for each $k \geq k_0$. Since $t_k \to 0$ as $k \to \infty$, this implies that $Z(t) \neq 0$ for each $t>0$.
}
\end{proof}

%

We now pass to the proof of non-uniqueness in the case of fBm with $H< \frac{1}{2}$.

\begin{proof}[Proof of Theorem \ref{thm:fbm} (2)]
Step 1. We first construct a.s. a solution $Z$ with $Z_0 = 0$ and $Z \neq 0$ on $(0,1]$. We fix a sequence of times $t_k \downarrow 0$, with $t_k-t_{k+1} \sim k^{-\gamma}$, where 
\begin{equation} \label{eq:cond1}
\gamma>1
\end{equation}
will be fixed later on.

We fix an initial condition $\ell_0 > 0$, $\delta_0=0$ and as in the previous proof we want to obtain an a.s. positive lower bound on $\ell(t)$ for $t>0$, which does not depend on $\ell_0$.

For $k \geq 1$ define the event
\[A_k := \Bigg\{  \mu \left(  \left\{ t \in [t_{k+1},t_k], \;\; \left( \lambda(t) - \min_{t_{k+1}\leq s \leq t} \lambda(s)\right) dt \geq k^{-\theta} \right\} \right) \geq k^{-\nu}, \Bigg\}\]
where $\mu$ denotes Lebesgue measure, for some $\theta, \nu$ satisfying
\begin{equation} \label{eq:cond2}
\nu > \frac{\theta}{H} > \gamma.
\end{equation}
Note that by the scaling property of fBm together with the small ball estimate Lemma \ref{lem:sbI} below, one has that $1 - \P^H(A_k) \leq c\exp(-c k^{\delta'})$ for some $\delta'>0$, so that by the Borel-Cantelli lemma, almost surely $A_k$ holds for $k$ large enough.

Now we claim that for $k$ large enough, it holds that for any solution $(\ell,\delta) $ to \eqref{eq:RRDEdt} with $\delta(0)=0$, 
\begin{equation*} \label{eq:claimAk}
\mbox{ on } A_k {\color{black} \cap \left\{ \sup_{s\leq t \in[0,t_k]} \cosh(|\lambda_{s,t}|)  \leq 2 \right\} } , \;\;\; \mbox{ } \ell(t_{k+1}) \geq (k+1)^{-\eta} \Rightarrow \ell(t_{k}) \geq k^{-\eta} ,
\end{equation*}
for suitably chosen $\eta$. Indeed, note that on $A_k$, if $\ell(t_{k+1}) \geq (k+1)^{-\eta}$, then
\begin{align*} \int_{t_{k+1}}^{t_k} \Gamma_0\left( \lambda_{0,\cdot} - 2  \int_0^\cdot \frac{du}{\ell_u}
\right)(s) ds &\geq  \left( k^{-\theta} -  2 k^{-\gamma}(k+1)^{\eta}\right) k^{-\nu} \end{align*}
so that
$$\ell(t_k) \geq (k+1)^{-\eta} + k^{-\theta -\nu} -2  k^{-\gamma + \eta -\nu} \geq {\color{black}   (k+1)^{-\eta} + \eta k^{-\eta-1} }  \geq k^{-\eta}$$
for $k$ large enough as long as
\begin{equation} \label{eq:cond3}
\theta + \nu < \eta+1, \;\;  \theta < \gamma- \eta.
\end{equation}
To conclude, it suffices to remark that if $H \in (0,1/2)$, it is possible to find $\gamma, \eta, \nu, \theta$ satisfying \eqref{eq:cond1},\eqref{eq:cond2} and \eqref{eq:cond3} (take $\eta \in (H,1-H)$ and $\gamma = 1+\epsilon$, $\theta = H(1+2\epsilon)$, $\nu = 1+3 \epsilon$ for small positive $\epsilon$).

\smallskip
Hence almost surely, we have a lower bound on $\ell(t_k)$ which does not depend on the initial condition. As in the proof of Theorem \ref{thm:cedt}, this implies the existence of a solution $Z^0$ with $Z^0(0)= 0$ but $Z^0(t) \neq 0$ for each $t>0$.

Step 2. We then show that almost surely, we may apply the previous construction to escape from $0$ at each $t \in [0,1]$.  To that end, take $\gamma, \theta, \nu$ as before and let 
\[\tilde{A}_k := \Bigg\{ \forall 0 \leq s \leq 1, \;\;\; \mu\left( \left\{ t \in [s, s+k^{-\gamma}], \;\; \left( \lambda(t) - \min_{s\leq u\leq t} \lambda(u)\right) dt \geq k^{-\theta} \right\} \right) \geq k^{-\nu}, \Bigg\},\]
Note that
$\tilde{A}_k \subset \cap_{j=0}^{ 2 k^\gamma} \hat{A}_{j,k}$,
where
\[ \hat{A}_{k,j} :=  \Bigg\{  \mu\left( \left\{ t \in [s_{j,k},s_{j+1,k}], \;\; \left( \lambda(t) - \min_{s_{j,k} \leq u\leq t} \lambda(u)\right) dt \geq k^{-\theta} \right\} \right) \geq k^{-\nu}, \Bigg\},\]
\[ s_{j,k} = j  \cdot \frac{k^{-\gamma}}{2}, j \geq 0.\]
By Lemma \ref{lem:sbI}, one has that $1 - \P^H(\tilde{A}_k) \leq  c k^\gamma \exp(-c k^{\delta''})$ for some $\delta''>0$, so that again by Borel-Cantelli, almost surely $\tilde{A}_k$ holds for $k$ large enough. This means that almost surely we may apply the construction of Step 1. at each $t \in [0,1)$ simultaneously to obtain solutions with $Z^t(u) = 0$ iff $u \in [0,t]$.\end{proof}

\section{Gaussian computations}

\subsection{Proof of \eqref{eq:asSN}} \label{sec:ProofSN}~\\

We start with some notations. We recall that the fractional Brownian motion of Hurst index $H \in (0,1)$ is a continuous centered Gaussian process $(B_t)_{t\in[0,1]}$ with covariance function given by
$$R(s,t) = \E \left[ B_s B_t \right]= \frac{1}{2} \left( t^{2H} + s^{2H} -  |t-s|^{2H}\right).$$
We consider the Hilbert space $\cH$ obtained by completing linear combinations of indicator functions of intervals under the norm induced by the scalar product $\left\langle\cdot,\cdot\right\rangle_{\cH}$, where 
$$\left\langle 1_{[0,s]}, 1_{[0,t]} \right\rangle_{\cH} = R(s,t) .$$

We then define smooth random variables as variables of the form $F = f\left(B(t_1), \ldots, B(t_k)\right)$ for some smooth scalar function $f$ with bounded derivatives, and for such $F$, we define its Malliavin derivative $DF$ as the $\cH$-valued random variable
$$DF = \sum_{j=1}^k \frac{\partial f}{\partial x_j} \left(B(t_1), \ldots, B(t_k)\right) 1_{[0,t_j]}. $$
We let $\mathbb{D}^{1,2}$ be obtained by completing smooth random variables under the norm $\| \cdot \|_{\mathbb{D}^{1,2}}$ where
$$\left\| F \right\|_{\mathbb{D}^{1,2}}^2 := \E \left[ | F |^2 + \left\|DF\right\|^2_{\cH} \right],$$
and note that $D$ may be extended to $\mathbb{D}^{1,2}$ (with values in $L^2$) continuously.

We will use Gaussian concentration of measure in the following form due to \"Ust\"unel (see \cite[Theorem VIII.1]{Ust95}).

\begin{proposition} \label{prop:ust}
Let $F \in \mathbb{D}^{1,2}$ be such that $\|DF\|_{\cH} \leq M < \infty$ almost surely. Then it holds that
\begin{equation*}
\forall x \in \R, \;\;\;\P\left( \left|F-\E[ F] \right| \geq  x \right) \leq 2 \exp\left(-\frac{x^2}{2M^2}\right).
\end{equation*}
\end{proposition}

\begin{corollary} \label{lem:conc2}
Let $G \in \mathbb{D}^{1,2}$, $G> 0$ a.s. with $\E[G] \leq 1$, and assume that there exists $M <\infty$ s.t.
$$ \P-\mbox{a.s.}, \;\;\; \left\|DG\right\|^2_{\cH} \leq M  G.$$
Then it holds that 
\begin{equation*}
\forall x >0, \;\;\;\P\left( \left|G-\E[ G] \right| \geq  x \right) \leq C \exp\left(-\frac{x}{C M}\right) + C \exp\left(-\frac{x^2}{C M}\right) 
\end{equation*}
for some constant $C>0$.
\end{corollary}

\begin{proof}
Applying Proposition \ref{prop:ust} to $F=\sqrt{G}$, we see that 
$$\P\left( \left|F-\E[F] \right| \geq  x \right) \leq 2 \exp\left(-2\frac{x^2}{M}\right).$$
In turn, this implies that
$$\E[G] \leq (\E[F])^2 + C M$$
and we have
$$\left|G-\E[ G] \right| \leq \left|F-\E[ F] \right|^2 + 2 \E[F]  \left| F-\E[F] \right | + C M $$
so that 
$$\P\left( \left|G-\E[ G] \right| \geq  x \right) \leq  \P\left(\left|F-\E[ F] \right|^2 \geq \frac{x}{2} - CM \right) +\P\left(\left|F-\E[ F] \right| \geq \frac{x}{4}  \right) $$
and the result follows.

\end{proof}

We can now proceed to the proof of \eqref{eq:asSN}. Recall that we have $t_k$ decreasing with \\$\delta t_k:=t_{k}-t_{k+1} \sim k^{-\alpha}$ for
\begin{equation*}
\max\left(\frac{5}{12H} ,1 \right)< \alpha < \frac{1}{2H},
\end{equation*}
and  
$$\delta_k = B(t_k) - \min_{[t_{k+1},t_k]} B, \;\; k\geq 0$$
$$S_N  = \sum_{k=0}^N \delta_k^2, \;\;\;\;\;\tilde{S}_N = S_N - \E[S_N].$$

We then claim that 
\begin{equation} \label{eq:concS}
\exists C>0, \;\;\; \forall N \leq M \leq N+N^{2\alpha H}, \forall x \leq 1, \P\left(\left|\tilde{S}_N - \tilde{S}_M\right| \geq x \right) \leq C \exp\left( - C^{-1} N^{2\alpha H - 1/2} x^2\right).
\end{equation}

We first note that for each $k$, $\delta_k$ is in $\D^{1,2}$ with $ D \delta_k = 1_{[t_k^\ast, t_k]}$
where $t_k^\ast \in [t_{k+1},t_k]$ is the (a.s. unique) time where $B$ attains its minimum on $[t_{k+1},t_k]$, cf . \cite{LN03}. It follows that
\begin{align*} 
\left\| D \left(S_M - S_N\right)\right\|_{\cH}^2 &= \sum_{N < k,j \leq M} \delta_k \delta_j \left\langle 1_{[t_k^\ast, t_k]},1_{[t_j^\ast, t_j]}\right\rangle_{\cH} \\
&\leq  \left(S_M - S_N\right)  \left( \sum_{N < k,j \leq M}  \left\langle 1_{[t_k^\ast, t_k]},1_{[t_j^\ast, t_j]}\right\rangle^2_{\cH} \right)^{1/2}
\end{align*}
by the Cauchy-Schwarz inequality. We now bound the sum appearing on the right-hand side. We will write $a \lesssim b$ when $a \leq C b$ for some constant $C$. Note that for $k \geq j+2$, 
\begin{align*}
2 \left\langle 1_{[t_k^\ast, t_k]},1_{[t_j^\ast, t_j]}\right\rangle_{\cH}  &=  |t_j - t_{k}|^{2H} + |t_j^\ast - t_{k}^\ast|^{2H} -|t_j - t_{k}^\ast|^{2H} - |t_j^\ast - t_{k}|^{2H} \\
&\lesssim \int_{t_j^\ast}^{t_j} \left((t_k^\ast-s)^{2H-1} - (t_k - s)^{2H-1} \right) ds \\
&\lesssim (t_k-t_k^\ast) (t_j - t_j^\ast) \left(t_{j+1} - t_k \right)^{2H-2} \\
&\lesssim k^{-\alpha} j^{-\alpha} \left( j^{1-\alpha} - k^{1-\alpha} \right)^{2H-2}.
\end{align*}
We then have
\begin{align}
 \sum_{N \leq k,j < M}  \left\langle 1_{[t_k^\ast, t_k]},1_{[t_j^\ast, t_j]}\right\rangle^2_{\cH}  &\lesssim \sum_{N \leq k} (\delta t_k)^{4H} + \sum_{N \leq j} j^{-2\alpha}\sum_{j+1 < k}  k^{-2\alpha} \left( j^{1-\alpha} - k^{1-\alpha} \right)^{4H-4}. \label{eq:DS}
\end{align}
The first sum is of order $N^{1-4H\alpha}$, whereas the second sum we estimate by splitting the inner sum in two sums depending on whether $k \leq 2j$ or $k>2j$. We have
\begin{align*}
\sum_{j< k \leq 2j}\left(j^{1-\alpha}-k^{1-\alpha}\right)^{4H-4} k^{-2\alpha} &\lesssim j^{-2\alpha} \sum_{j< k \leq 2j} k^{-\alpha(4H-4)} (k-j)^{4H-4} \\
&\lesssim j^{-2\alpha} j^{-\alpha(4H-4)} j^{4H-3} = j^{2 \alpha -4H (\alpha-1) - 3}
\end{align*}
whereas 
\begin{align*}
\sum_{2j< k}\left(j^{1-\alpha}-k^{1-\alpha}\right)^{4H-4} k^{-2\alpha} &\lesssim j^{(1-\alpha)(4H-4)} j^{1-2\alpha} = j^{2\alpha -4H (\alpha-1) - 3}
\end{align*}
so that the second sum in \eqref{eq:DS} admits an upper bound of order $\sum_{j\geq N} j^{-4H (\alpha-1) - 3} \lesssim N^{-4H (\alpha-1)  -2} =~o(N^{1-4H\alpha})$, and we obtain
\[
\left( \sum_{N < k,j \leq M}  \left\langle 1_{[t_k^\ast, t_k]},1_{[t_j^\ast, t_j]}\right\rangle^2_{\cH}\right)^{1/2} \lesssim N^{1/2-2H\alpha}.
\]

Note that $\E[S_M-S_N] \sim c (M^{1-2\alpha H} - N^{1-2\alpha H})$ is bounded if $M \leq N + N^{2\alpha H}$, we can therefore apply Corollary \ref{lem:conc2} to obtain \eqref{eq:concS}.

Then, we let $N_k = k^{\frac{1}{1-2\alpha H}}$ and note that for $N_k \leq N \leq N_{k+1}$
$$\P\left( \left|\tilde{S}_{N_{k+1}} - \tilde{S}_{N} \right| \geq k^{-\gamma} \right) \leq C \exp\left(-C^{-1} k^{-2\gamma + \frac{2\alpha H -1/2}{1-2\alpha H}}\right)$$
so that if 
$$\frac{2\alpha H -1/2}{1-2\alpha H} > 2$$
namely if $ \alpha > \frac{5}{12H},$
one may choose some $\gamma>1$ for which it holds that
$$\sum_{k} \sum_{N_k \leq N \leq N_{k+1}} \P\left( \left|\tilde{S}_{N_{k+1}} - \tilde{S}_{N} \right| \geq k^{-\gamma} \right) < \infty,$$
and we conclude by the Borel-Cantelli lemma.

\vspace{5mm}

\subsection{Small ball properties}~\\

We start by a small ball estimate for reflected fBm in supremum norm.

\begin{lemma} \label{lem:sb}
Let $B$ be the fractional Brownian motion of Hurst index $H \in{\color{black} (0,1/2]}$, then there exists $c>0$ s.t. for all $x > 0$,
\begin{equation}
\P \left( \max_{t \in [0,1]} \left( B(t) - \min_{0\leq s \leq t} B(s)\right) \leq x \right) \leq c \exp\left(-c x^{-1/H}\right) .
\end{equation}
\end{lemma}

\begin{proof}
The proof follows by considering increments over small intervals as in \cite{MR95}. Namely, given $x>0$ we let $x^{-1/H} \leq n \leq x^{-1/H} +1$ and obtain
\begin{align*}
\P \left( \sup_{t \in [0,1]} \left( B(t) - \min_{0\leq s \leq t} B(s)\right) \leq x \right) &\leq \P \left(\forall 1\leq k \leq n,  B\left(\frac{k}{n}\right) - B\left(\frac{(k-1)}{n}\right) \leq x \right) \\
&\leq \Pi_{k=1}^n \P \left(B\left(\frac{k}{n}\right) - B\left(\frac{(k-1)}{n}\right) \leq x \right) \\
& \leq \left(\int_{-\infty}^{C} \frac{e^{-y^2/2}  dy}{\sqrt{2\pi}} \right)^n \\
&\leq c \exp(-c x^{-1/H}),
\end{align*}
where $C,c$ only depend on $H$, {\color{black}where we have used Slepian's lemma in the second inequality (recall that fBm has negatively correlated increments for $H \leq \frac 1 2$), and fBm scaling in the third inequality.}
\end{proof}

The previous lemma allows us to deduce another small ball bound on $L^1$-type information. We do not expect this result to be sharp but it will be sufficient for our purposes. 

\begin{lemma} \label{lem:sbI}
Let $B$ be the fractional Brownian motion of Hurst index $H \in (0,1/2]$, then for all $0<\kappa < H$ there exists $c>0$ s.t. for all $x, y> 0$,
\begin{equation*} \label{eq:sbI}
\P \Bigg( \mu\left(  \left\{ t \in [0,1], \;\; \left( B(t) - \min_{0\leq s \leq t} B(s)\right)  \geq x \right\} \right) \leq y \Bigg) \leq c \exp\left(-c x^{-1/H}\right) + c \exp\left( -c x^2 y^{-2\kappa}\right) ,
\end{equation*}
where $\mu$ denotes Lebesgue measure.
\end{lemma}

\begin{proof}
Fix $0< \kappa < H$ and note that
\begin{align*} \max_{t \in [0,1]} \left( B(t) - \min_{0\leq s \leq t} B(s) \right) \geq 2x, &\;\; \|B\|_{\kappa} \leq C \\
& \Rightarrow  \mu\left(  \left\{ t \in [0,1], \;\; \left( B(t) - \min_{0\leq s \leq t} B(s)\right) dt \geq x \right\} \right) \geq x^{1/\kappa} C^{-1/\kappa}, 
\end{align*}
where $\|B\|_{\kappa}$ is the $\kappa$-H\"older norm of $B$ on $[0,1]$. Hence taking $C=x  y^{-\kappa}$, and using Lemma \ref{lem:sb} as well as Gaussian tails of $\|B\|_{\kappa}$, we obtain the result.
\end{proof}

\bibliographystyle{plain}
\bibliography{reflected}

\begin{thebibliography}{10}

\bibitem{Aida15}
Shigeki Aida.
\newblock Reflected rough differential equations.
\newblock {\em Stochastic Process. Appl.}, 125(9):3570--3595, 2015.

\bibitem{Aida16}
Shigeki {Aida}.
\newblock {Rough differential equations containing path-dependent bounded
  variation terms}.
\newblock {\em arXiv e-prints}, page arXiv:1608.03083, Aug 2016.

\bibitem{CMR17}
Charles {Castaing}, Nicolas {Marie}, and Paul {Raynaud De Fitte}.
\newblock {Sweeping Processes Perturbed by Rough Signals}.
\newblock {\em arXiv e-prints}, page arXiv:1702.06495, Feb 2017.

\bibitem{DGHT19}
Aur\'{e}lien Deya, Massimiliano Gubinelli, Martina Hofmanov\'{a}, and Samy
  Tindel.
\newblock One-dimensional reflected rough differential equations.
\newblock {\em Stochastic Process. Appl.}, 129(9):3261--3281, 2019.

\bibitem{FS15}
Adrian Falkowski and Leszek S\l{}omi\'{n}ski.
\newblock Stochastic differential equations with constraints driven by
  processes with bounded {$p$}-variation.
\newblock {\em Probab. Math. Statist.}, 35(2):343--365, 2015.

\bibitem{FS15b}
Adrian {Falkowski} and Leszek {S\l{}omi\'{n}ski}.
\newblock {Sweeping processes with stochastic perturbations generated by a
  fractional Brownian motion}.
\newblock {\em arXiv e-prints}, page arXiv:1505.01315, May 2015.

\bibitem{FS17}
Adrian Falkowski and Leszek S\l{}omi\'{n}ski.
\newblock S{DE}s with constraints driven by semimartingales and processes with
  bounded {$p$}-variation.
\newblock {\em Stochastic Process. Appl.}, 127(11):3536--3557, 2017.

\bibitem{Fla11}
Franco Flandoli.
\newblock {\em Random perturbation of {PDE}s and fluid dynamic models}, volume
  2015 of {\em Lecture Notes in Mathematics}.
\newblock Springer, Heidelberg, 2011.
\newblock Lectures from the 40th Probability Summer School held in Saint-Flour,
  2010, \'{E}cole d'\'{E}t\'{e} de Probabilit\'{e}s de Saint-Flour.
  [Saint-Flour Probability Summer School].

\bibitem{FV}
Peter~K. Friz and Nicolas~B. Victoir.
\newblock {\em Multidimensional stochastic processes as rough paths}, volume
  120 of {\em Cambridge Studies in Advanced Mathematics}.
\newblock Cambridge University Press, Cambridge, 2010.
\newblock Theory and applications.

\bibitem{LN03}
Noureddine Lanjri~Zadi and David Nualart.
\newblock Smoothness of the law of the supremum of the fractional {B}rownian
  motion.
\newblock {\em Electron. Comm. Probab.}, 8:102--111, 2003.

\bibitem{LS84}
P.-L. Lions and A.-S. Sznitman.
\newblock Stochastic differential equations with reflecting boundary
  conditions.
\newblock {\em Comm. Pure Appl. Math.}, 37(4):511--537, 1984.

\bibitem{Lyo98}
Terry~J. Lyons.
\newblock Differential equations driven by rough signals.
\newblock {\em Rev. Mat. Iberoamericana}, 14(2):215--310, 1998.

\bibitem{MR95}
Ditlev Monrad and Holger Rootz\'{e}n.
\newblock Small values of {G}aussian processes and functional laws of the
  iterated logarithm.
\newblock {\em Probab. Theory Related Fields}, 101(2):173--192, 1995.

\bibitem{RTT19}
Alexandre {Richard}, Etienne {Tanr{\'e}}, and Soledad {Torres}.
\newblock {Penalisation techniques for one-dimensional reflected rough
  differential equations}.
\newblock {\em arXiv e-prints}, page arXiv:1904.11447, Apr 2019.

\bibitem{Sai87}
Yasumasa Saisho.
\newblock Stochastic differential equations for multidimensional domain with
  reflecting boundary.
\newblock {\em Probab. Theory Related Fields}, 74(3):455--477, 1987.

\bibitem{Tan79}
Hiroshi Tanaka.
\newblock Stochastic differential equations with reflecting boundary condition
  in convex regions.
\newblock {\em Hiroshima Math. J.}, 9(1):163--177, 1979.

\bibitem{Ust95}
Ali~S\"{u}leyman \"{U}st\"{u}nel.
\newblock {\em An introduction to analysis on {W}iener space}, volume 1610 of
  {\em Lecture Notes in Mathematics}.
\newblock Springer-Verlag, Berlin, 1995.

\end{thebibliography}
\end{document}